\documentclass[11pt,a4paper,reqno]{article}

\usepackage{amssymb}
\usepackage{latexsym}
\usepackage{amsmath}
\usepackage{graphicx}
\usepackage{amsthm}
\usepackage{empheq}
\usepackage{bm}
\usepackage{booktabs}
\usepackage[dvipsnames]{xcolor}
\usepackage{pagecolor}
\usepackage{subcaption}
\usepackage{tikz,lipsum,lmodern}
\usepackage{enumitem}
\usepackage{calligra}
\usepackage{mathrsfs}
\usepackage[margin=3cm]{geometry}
\usepackage{authblk}
\usepackage{enumitem}
\setitemize{itemsep=0em}
\setenumerate{itemsep=0em}
\usepackage{hyperref}
\hypersetup{
  colorlinks   = true, 
  urlcolor     = blue, 
  linkcolor    = purple, 
  citecolor   = ForestGreen 
}

\numberwithin{equation}{section}

\theoremstyle{definition}
\newtheorem{theorem}{Theorem}[section]
\newtheorem{corollary}[theorem]{Corollary}
\newtheorem{proposition}[theorem]{Proposition}
\newtheorem{definition}[theorem]{Definition}

\newtheorem{notation}[theorem]{Notation}
\newtheorem{remark}[theorem]{Remark}

\newcommand{\numberset}{\mathbb}

\newcommand{\Z}{\numberset{Z}}

\newcommand{\C}{\mathcal{C}}

\newcommand{\F}{\numberset{F}}

\newcommand{\cL}{\mathcal{L}}
\newcommand{\fq}{\F_q}
\newcommand{\fqn}{\F_{q^n}}
\newcommand{\fqm}{\F_{q^m}}

\newcommand{\mL}{\mathscr{L}}

\newcommand{\mC}{\mathcal{C}}

\newcommand{\rk}{\textnormal{rk}}

\renewcommand{\longrightarrow}{\to}

\newcommand{\aut}{\textnormal{Aut}}

\newcommand{\drk}{d^{\textnormal{rk}}}

\newcommand{\GL}{\textnormal{GL}_m(q)}
\newcommand{\PG}{\textnormal{PG}}

\def\cov{\mathrel{<\kern-.6em\raise.015ex\hbox{$\cdot$}}}
\def\<{\left<}
\def\>{\right>}

\newcommand{\qqbin}[2]{\begin{bmatrix}{#1}\\ {#2}\end{bmatrix}_q}
\newcommand{\qmqbin}[2]{\begin{bmatrix}{#1}\\ {#2}\end{bmatrix}_{q^m}}

\newcommand{\supp}{{\textnormal{supp}}}

\newcommand*{\myproofname}{Proof of the claim}

\allowdisplaybreaks

\title{\textbf{Whitney Numbers of Rank-Metric Lattices \\ and Code Enumeration}}

\usepackage{setspace}
\author[1]{Giuseppe Cotardo\thanks{G.C. was partially supported by the NSF grants DMS-2037833 and DMS-2201075, and by the Commonwealth Cyber Initiative.}}
\affil[1]{Virginia Tech, Blacksburg, U.S.A.}

\author[2]{Alberto Ravagnani\thanks{A. R. is supported by the Dutch Research Council through grants OCENW.KLEIN.539 and VI.Vidi.203.045.}}
\affil[2]{Eindhoven University of Technology, the Netherlands}

\author[3]{Ferdinando Zullo\thanks{F. Z. was partially supported by the project COMBINE from University of Campania and was partially supported by the Italian National Group for Algebraic and Geometric Structures and their Applications (GNSAGA - INdAM).}}
\affil[3]{Universit\`a degli Studi della Campania ``Luigi Vanvitelli'', Caserta, Italy.}

\date{}

\usepackage{setspace}
\begin{document}
\maketitle
	
\abstract{
We investigate the Whitney numbers of the first kind of rank-metric lattices, which are closely linked to the open problem of enumerating rank-metric codes having prescribed parameters. 
We apply methods from the theory of hyperovals and linear sets to compute these Whitney numbers for infinite families of rank-metric lattices. As an application of our results, we prove asymptotic estimates on the density function of certain rank-metric codes that have been conjectured in previous work. 
}

\medskip

\section*{Introduction}

This paper applies 
methods from finite geometry to an open order theory problem, solving some new instances. The problem is the computation of the Whitney numbers (of the first kind) of rank-metric lattices.

Rank-metric lattices were introduced in~\cite{cotardo2023rank} in connection with the problem of enumerating rank-metric codes with correction capability bounded from below~\cite{cooperstein1998external,delsarte1978bilinear,gabidulin1985theory,roth1991maximum}. They 
can be seen as the $q$-analogues of higher-weight Dowling lattices, introduced by Dowling in~\cite{dowling1971codes,dowling1973q} in connection with the \textit{packing problem} for the Hamming metric, which is the central question of classical coding theory~\cite{dowling1971codes}.

Rank-metric lattices are indexed by 
four parameters. More precisely,
$\mL_i(n,m;q)$ is the subgeometry
of the subspaces of $\F_{q^m}^n$ generated by vectors of rank weight up to~$i$; see Section~\ref{sec:1} for a precise definition. 
In~\cite{cotardo2023rank}, the supersolvable rank-metric lattices were completely characterized, and their 
Whitney numbers were computed for some sporadic parameter sets. 
The Whitney numbers of $\mL_2(4,4;2)$ were determined under the assumption that they are polynomial expressions in the field size $q$. The techniques applied in~\cite{cotardo2023rank} do not seem to extend naturally to other rank-metric lattices.

In this paper, we propose a new approach for computing the Whitney numbers of rank-metric lattices, which leverages on techniques from finite geometry. We use the theory of hyperovals and linear sets, as well as the description of rank-metric codes as spaces of linearized polynomials.  This allows us to compute the Whitney numbers of rank-metric lattices in several new instances, and to confirm the value conjectured in~\cite[Section~6]{cotardo2023rank} for the lattice~$\mL_2(4,4;2)$ under the polynomial assumption. Using the same methods, we also establish new results on the density function of rank-metric codes, which is another wide open problem in coding theory~\cite{antrobus2019maximal,byrne2020partition,gruica2022common,gruica2023rank,gruica2024densities,neri2018genericity}.

The remainder of this paper is organized as follows. In Section~\ref{sec:1} we recall the definition of 
rank-metric lattices and outline the motivation for this work. In Section~\ref{sec:2}, we introduce the tools we will use throughout the paper: linear sets, hyperovals, rank-metric codes, and linearized polynomials. In Section~\ref{sec:3}, we establish the classification of two-dimensional MRD codes in  $\F_{q^4}^4$ and $\F_{2^m}^m$, which we explicitly compute in Section~\ref{sec:4}. In Section~\ref{sec:5}, we count one-weight rank-metric codes using their multivariate polynomial description. Finally, in Section~\ref{sec:6}, we provide a recursive formula for computing the Whitney numbers of the first kind and give a closed formula for the ones of $\mL_2(n,3;q)$ with $n \in \{4,5,6\}$.

\section{Motivation and Preliminaries}
\label{sec:1}

We start by recalling the definition of a lattice. We refer the reader to~\cite{stanley2011enumerative} for the necessary background on order theory. We follow the notation of~\cite{stanley2011enumerative}.

\begin{definition}
	A \textbf{finite lattice} $(\mL, \le)$ is a finite poset where every $s,t \in \mL$ have a join and a meet.
	In this case, join and meet can be seen as commutative and associative operations
	$\vee, \wedge : \mL \times \mL \to \mL$.
	In particular, the \textbf{join} (resp., the \textbf{meet}) of a non-empty subset $S \subseteq \mL$ is well-defined as the join (resp., the meet) of its elements and denoted by~$\vee S$ (resp., by~$\wedge S$). Furthermore, $\mL$ has a minimum and maximum element ($0$ and~$1$, resp.). Finally, if $\mL$ is classified, then \textbf{ rank} of $\mL$ is~$\rk(\mL)=\rho(1)$, where $\rho$ is the rank function of $\mL$.
\end{definition}

\begin{definition}
    Let $\mL$ be a finite, graded lattice with rank function $\rho$ and M\"obius function~$\mu$. The \textbf{characteristic polynomial} of $\mL$ is the element of $\Z[\lambda]$ defined as
    \begin{equation*}
        \chi(\mL;\lambda):=\sum_{s\in\mL}\mu(s)\lambda^{\rk(\mL)-\rho(s)}=\sum_{i=0}^{\rk(\mL)}w_i(\mL)\lambda^{\rk(\mL)-i},
    \end{equation*}
    where
    \begin{equation*}
        w_i(\mL)=\sum_{\substack{s\in\mL\\\rho(s)=i}}\mu(s)
    \end{equation*}
    is the $i$-\textbf{th Whitney number of the first kind} of $\mL$. The  $i$-\textbf{th Whitney number of the first kind} of $\mL$, denoted by $W_i(\mL)$, is the number of elements of $\mL$ of rank~$i$.
\end{definition}

The family of \textit{rank-metric lattices} was introduced in~\cite{cotardo2023rank}, in connection with the problem of enumerating rank-metric codes with prescribed parameters, a current open problem in coding theory~\cite{antrobus2019maximal,byrne2020partition,gruica2022common,gruica2023rank}.
We recall their definition after establishing the notation for the rest of the paper.

In the sequel, $q$ is a prime power and $\fq$ is the finite with $q$ elements. 
We let $m,n$ be positive integers with $m \ge n$.

\begin{definition}
	The \textbf{rank} (\textbf{weight}) of a vector $v\in\fqm$ is the dimension over $\fq$ of the $\fq$-span of its entries.
	We denote it by $\rk(v)$. The (\textbf{rank}) \textbf{distance} between vectors $v,w \in \F_{q^m}^n$ is $d(v,w)=\rk(v-w)$.
    A \textbf{rank-metric code} is a $k$-dimensional $\fqm$-linear subspace $C \le \fqm^n$ with $k \ge 1$. The \textbf{minimum distance} of $C$ is $d(C):=~\min\{\rk(v):v\in\fqm^n \textup{ with } v\neq 0\}$. We refer to $C$ as a $[n,k,d(C)]_{\fqm}$-code.
\end{definition}

Rank-metric lattices are defined as follows. 

\begin{definition}
    Let $i\in \{1, \ldots n\}$. The \textbf{rank-metric lattice} (\textbf{RML} in short) $\mL_i(n,m;q)$ associated with the $4$-tuple $(i,n,m,q)$ is the geometric lattice whose atoms are the $1$-dimensional $\fqm$-linear subspaces of $\fqm^n$ generated by the vectors of rank at most $i$, i.e.,
    \begin{align*}
        \mL_i(n,m;q):=
        \{\<v_1,\ldots,v_\ell\>_{\fq} \colon \ell \ge 1, \, v_1,\ldots,v_\ell\in T_i(n,m;q)\},
    \end{align*}
    where $T_i(n,m;q):=\{v\in\fqm^n:\rk(v)\leq i\}$. The order is given by the inclusion of $\fqm$-linear subspaces of $\fqm^n$. 
\end{definition}

This paper focuses on the computation of the Whitney numbers of the 
first kind of rank-metric lattices. We follow the notation in~\cite{cotardo2023rank} and denote by $w_j(i,n,m;q)$ be the $j$-th Whitney number of the first kind of the lattice $\mL_i(n,m;q)$. We apply various techniques, ranging from coding theory, order theory, and finite geometry (hyperovals and linear sets).

\section{Tools} \label{sec:2}
In this section we recall some notions and results that will be needed throughout the paper. The section is divided into three subsections. 

\subsection{Hyperovals and Linear Sets}
Let PG$(r,q)$ denotes the Desarguesian projective space of dimension $r$ of order $q$ and let $\PG(V,\fq)$ denotes the projective space obtained by $V$, for some $\fq$-linear vector space $V$.
Classical references are~\cite{hirschfeld1979projective,lavrauw2015field,polverino2010linear}.

\begin{definition}
A set $\mathcal{O}$ of $q+1$ points in $\mathrm{PG}(2,q)$ is called an \textbf{oval} if no three of its points are collinear.
Moreover, we say that the ovals $\mathcal{O}_1$ and $\mathcal{O}_2$ are \textbf{equivalent} if there exists a collineation $\phi$ of the plane such that $\phi(\mathcal{O}_1)=\phi(\mathcal{O}_2)$.
\end{definition}

Observe that conics are a family of ovals. Furthermore, by the Segre's Theorem we know that when $q$ is odd, all the ovals are equivalent to a conic; see~\cite{segre1955ovals}.

When $q$ is even, every oval can be uniquely extended to a set of $q+2$ points satisfying again the condition that no three of its points are still collinear. These sets are known as \textbf{hyperovals}.
We refer to~\cite{caullery2015classification,cherowitzo1996hyperovals,de2002arcs} for a list of known examples and classification results.
We will be mainly interested in the case of translation hyperovals.

\begin{definition}
Let $\mathcal{H}$ be a hyperoval in $\PG(2,q)$, with $q$ even. We say that $\mathcal{H}$ is a \textbf{translation hyperoval} if there exists a secant line $\ell$ to $\mathcal{H}$ such that the group of elations having axis $\ell$ acts transitively on the points  of $\mathcal{H}\setminus \ell$.
\end {definition}

The set of points of any hyperoval $\mathcal{H}$ in $\mathrm{PG}(2,q)$ is
\[ \{ (x,f(x),1) \colon x \in \mathbb{F}_q \} \cup \{(1,0,0),(0,1,0)\}, \]
where $f \in \mathbb{F}_q[x]$, up to equivalence.
We will denote such a set of points by $\mathcal{H}_f$.
It is not difficult to see that $\mathcal{H}_f$ is a translation hyperoval if and only if $f$ is an additive permutation polynomial and $f(x)/x$ is a permutation polynomial in $\F_{q}\setminus\{0\}$ as well.
Translation hyperovals have been completely classified.

\begin{theorem} (\cite[Main Theorem]{payne1971complete} and~\cite{hirschfeld1975ovals})\label{th:classhyper}
Let \smash{$f \in \mathbb{F}_q[x]$} with $q=2^h$. Then $\mathcal{H}_f$ is a translation hyperoval if and only if \smash{$f(x)=ax^{2^j}$} for some $j \in \{1,\ldots,h-1\}$ with $\gcd(h,j)=1$.
\end{theorem}

In this paper, we mainly focus on $\fq$-linear sets in the projective line.

\begin{definition}\label{def:linset}
Let $V$ be an $\fqm$-linear vector space of dimension 2, let $U $ be an $\fq$-linear subspace of $V$, and let $\Lambda=\PG(V,\fqm)=\PG(1,q^m)$. The set
\[ L_U=\{\langle \mathbf{u} \rangle_{\fqm} \, \colon \,  \mathbf{u}\in U\setminus\{\mathbf{0}\}\} \]
is called an $\fq$-\textbf{linear set} of \textbf{rank} $\dim_{\fq}(U)$.
\end{definition}

The following is an important concept related to linear sets.

\begin{definition}\label{def:weight}
The \textbf{weight} of a point $P=\langle \mathbf{v}\rangle_{\fqm} \in \Lambda$ in $L_U$ is $\dim_{\fq}(U \cap \langle \mathbf{v}\rangle_{\fqm})$.
\end{definition}

The maximum value for $|L_U|$ can be reached in the case where all the points of $L_U$ have weight one, therefore if $L_U$ is any $\fq$-linear set of rank $k$ we have 
\begin{equation}\label{eq:card}
    |L_U| \leq \frac{q^k-1}{q-1}.
\end{equation}

Linear sets attaining the bound in \eqref{eq:card} are called \textbf{scattered} and the $\fq$-linear subspace~$U$ is called a \textbf{scattered subspace}.
They were originally introduced in~\cite{blokhuis2000scattered}.
It is difficult, in general, to determine whether or not two linear sets are equivalent. Recall that~$L_{U_1}$ and~$L_{U_2}$ are  \textbf{equivalent} if there exists $\varphi \in \mathrm{P\Gamma L}_1(q^m)$ with $\varphi(L_{U_1})=~L_{U_2}$.
In~\cite{csajbok2018classes}, the authors introduced the following concept for detecting families of linear sets for which this problem is easier to address.

\begin{definition}\label{def:simple}
An $\fq$-linear set $L_U$ in $\PG(1,q^m)=\PG(V,\fqm)$ is \textbf{simple} if for each $\fq$-linear subspace $W$ of $V$, $L_W$ is $\mathrm{P\Gamma L}_2(q^m)$-equivalent to $L_U$ if and only if $U$ and~$W$ are in the same orbit of $\mathrm{\Gamma L}_2(q^m)$.
We say that $L_U$ has \textbf{$\mathrm{GL}$-class one} if for each $\fq$-linear subspace~$W$ of~$V$,  $L_W$ is $\mathrm{PG L}_2(q^m)$-equivalent to $L_U$ if and only if $U$ and $W$ are in the same orbit of~$\mathrm{G L}_2(q^m)$.
\end{definition}

\subsection{Polynomial Description of Rank-Metric Codes}

We denote by $\fq[x_1,\ldots,x_\ell]$ the ring of polynomials in the variables $x_1,\ldots,x_\ell$ with coefficient in $\fq$. We let $\mathrm{GL}_k(q)$ and $\mathrm{\Gamma L}_k(q)$ denote the general linear and the general semilinear groups, respectively.

We recall the connection between rank-metric codes and linearized polynomials. 
We start describing some properties of linearized polynomials and refer the reader to~\cite{wu2013linearized} for further details.

\begin{definition}
A (\textbf{linearized}) \textbf{$q$-polynomial} over $\F_{q^m}$ is a polynomial of the form
$$ f:=\sum_{i \ge 0}f_i x^{q^i} \in \F_{q^m}[x].$$
The largest $i$ with $f_i \neq 0$ is the \textbf{$q$-degree} of $f$.
The set of $q$-polynomials modulo~$x^{q^m}-x$ is denoted by $\mathcal{L}_{m,q}[x]$.
\end{definition}

Note that $\mathcal{L}_{m,q}[x]$ is an $\F_q$-algebra equipped with the operations of addition and composition of polynomials, and scalar multiplication by elements of $\F_q$.  The elements of~$\cL_{m,q}[x]$ are in one-to-one correspondence with the $q$-polynomials of $q$-degree upper bounded by $m-1$. 
Throughout the paper, we abuse notation and denote an element of $\cL_{m,q}[x]$ as its unique representative of $q$-degree at most $m-1$. 

\begin{remark} \label{rem:end}
We have $\F_q$-algebra automorphisms
\begin{equation}\label{eq:isom_sigma} 
(\mathcal{L}_{m,q}[x],+,\circ)\cong(\mathrm{End}_{\fq}(\mathbb{F}_{q^m}),+,\circ) \cong (\mathbb{F}_q^{m \times m},+,\cdot),
\end{equation} 
where in the 3-tuples we omitted the scalar multiplication by an element of~$\F_q$. 
We refer the reader to~\cite{wu2013linearized} and~\cite{gruica2023rank} for further details. 
\end{remark}

In the following, we consider $\mathcal{L}_{m,q}[x]$ endowed with the metric
\[ d_L(f,g)=\dim_{\fq}(\mathrm{Im}(f-g)), \]
for any $f,g \in \mathcal{L}_{m,q}[x]$.
It is not hard to check that $(\mathcal{L}_{m,q}[x],d_L)$ and $(\F_{q^m}^m,d)$ are isometric, where $d$ is the rank distance; see Section~\ref{sec:1}.
As a consequence, an $\F_{q^m}$-linear rank-metric code in $\F_{q^m}^m$ corresponds to an $\F_{q^m}$-linear subspace of $\mathcal{L}_{m,q}[x]$.
Therefore, we will sometimes say
``an $\F_{q^m}$-linear rank-metric code $\mC \le \mathcal{L}_{m,q}[x]$''.

Interpreting rank-metric codes as 
spaces of $q$-polynomials will allow us to use results on linear sets and to count certain families of MRD codes.

\begin{notation}
If \smash{$f=\sum_{i=0}^{m-1}f_ix^{q^i} \in \cL_{m,q}[x]$} and $\rho \in \aut(\F_q)$ is a field automorphism, then we let \smash{$f^\rho:= \sum_{i=0}^{m-1}\rho(f_i)x^{q^i}$}.
Furthermore, if \smash{$\mC \le \cL_{m,q}[x]$} is a rank-metric code and $\rho \in \aut(\F_q)$, we let $\mC^\rho:=\{f^\rho \mid f \in \mC\}$.
\end{notation}

We recall two notions of equivalence between rank-metric codes. 
\begin{definition}
We say that the $\mathbb{F}_{q^m}$-linear rank-metric codes $\C_1, \C_2 \le \cL_{m,q}[x]$ are \textbf{equivalent}
if there exist invertible \smash{$q$-polynomials} $f,g \in \cL_{m,q}[x]$ and a field automorphism \smash{$\rho \in \mathrm{Aut}(\fq)$} such that
\[ \C_1=f \circ \C_2^\rho \circ g. \]
The \textbf{automorphism group} of $\mC$ is
\begin{align*}
    \aut(\mC) = \{(f_1,\rho,f_2) \in 
    \GL \times \mathrm{Aut}(\fq) \times \GL
    \mid \C=f_1 \circ \C^\rho \circ f_2\},
\end{align*}
where $\GL$ denotes the group of invertible linearized polynomials in $\mathcal{L}_{m,q}[x]$.
\end{definition}
\begin{remark}
    Note that for any $\C\leq \mathcal{L}_{m,q}[x]$,  $f,g \in \cL_{m,q}[x]$ and a field automorphism \smash{$\rho \in \mathrm{Aut}(\fq)$}, then $f \circ \C^\rho \circ g$ may not be an $\F_{q^m}$-subspace of $\mathcal{L}_{m,q}[x]$.
\end{remark}

We aim to count the number of $\mathbb{F}_{q^m}$-linear rank-metric codes in $\mathbb{F}_{q^m}^n$ and therefore the natural concept of equivalence in this context makes use of $\F_{q^m}$-linear isometries of $(\mathcal{L}_{m,q}[x],d_L)$. The following is the second notion of equivalence.

\begin{definition}\label{def:equivrk}
We say that $\mathbb{F}_{q^m}$-linear rank-metric codes $\C_1, \C_2 \le \cL_{m,q}[x]$ are \textbf{linearly equivalent}
if there exists an invertible \smash{$q$-polynomial} $g \in \cL_{m,q}[x]$ such that
\[ \C_1=\C_2 \circ g. \]
The \textbf{linear automorphism group} of $\mC$ is
\begin{align*}
    \aut\mathrm{lin}(\mC) = \{g \in \mathrm{GL}_m(q) \mid \C=\C \circ g\}.
\end{align*}
\end{definition}
\begin{remark}
The linear automorphism group of $\mC$ is usually called \emph{right idealizer} or \emph{middle nucleus}; see~\cite{liebhold2016automorphism,lunardon2017kernels}.
\end{remark}

\subsection{Examples of MRD codes}

One of the most studied MRD codes are the \textbf{Gabidulin codes}. These were first introduced in~\cite{delsarte1978bilinear,gabidulin1985theory} and later generalized in~\cite{kshevetskiy2005new}. They are defined as follows. 

\begin{definition}\label{def:gengab}
Let $k,m,s$ be three positive integers with $k\leq m$ and $\gcd(s,m)=1$. Then a \textbf{generalized Gabidulin code} is
\begin{align*}
    \mathcal{G}_{k,s,m}=\langle x,x^{q^s},\ldots,x^{q^{s(k-1)}} \rangle_{\mathbb{F}_{q^m}}.
\end{align*}
\end{definition}

In~\cite[Theorem 5.4]{neri2020equivalence} (see also~\cite{schmidt2018number}), it has been proved that the number of linearly inequivalent generalized Gabidulin codes in $\mathcal{L}_{m,q}[x]$ is $\varphi(m)/2$, where $\varphi$ is the Euler's totient function.
Their automorphism group has been fully determined; see~\cite{liebhold2016automorphism,sheekey2016new}.

\begin{theorem}\cite[Theorem 4]{sheekey2016new}\label{th:autGab}
Let $q=p^h$, with $p$ a prime and $h$ any positive integer.
The automorphism group of $\mathcal{G}_{k,s,m}$ is the set of  3-tuples  $\smash{(ax^{q^i}, \rho, bx^{q^{m-i}})}$ with $\smash{a,b \in \F_{q^m}^{\times}}$, $0 \le i \le m-1$, and $\rho \in \aut(\F_q)$. In particular, we have  that $|\aut(\mathcal{G}_{k,s,m})| = hm(q^m-1)^2$.
\end{theorem}

We obtain the following by restricting to the linear automorphism group.

\begin{corollary}\label{cor:autGab}
The linear automorphism group of $\mathcal{G}_{k,s,m}$ is 
\[ \{ax \colon a \in \mathbb{F}_{q^m}^*\}. \]
\end{corollary}

Successively, Sheekey in~\cite{sheekey2016new} extended the above family of MRD codes (see also~\cite{lunardon2018generalized}). We recall their definition in the $\mathbb{F}_{q^m}$-linear setting.

\begin{definition}\label{def:gentwigab}
Let $k,m,s$ be three positive integers with $k\leq m$ and $\gcd(s,m)=1$ and let $\delta \in \mathbb{F}_{q^m}$ with $\mathrm{N}_{q^m/q}(\delta)\ne (-1)^{mk}$. Then a \textbf{generalized twisted Gabidulin code} is
\begin{align*}
    \mathcal{T}_{k,s,m}(\delta)=\langle x^{q^s},\ldots,x^{q^{s(k-1)}},x+\delta x^{q^{s(n-1)}} \rangle_{\mathbb{F}_{q^m}}.
\end{align*}
\end{definition}

Also, its automorphism group has been completely determined.

\begin{theorem}(\cite[Theorem 7]{sheekey2016new} and~\cite[Theorem 4.4]{lunardon2018generalized})\label{th:autgentwi}
Let $q=p^h$, with $p$ a prime and $h$ any positive integer.
The automorphism group of $\mathcal{T}_{k,s,m}(\delta)$ is the set of  3-tuples  $\smash{(ax^{q^i}, \rho, bx^{q^{m-i}})}$ with $\smash{a,b \in \F_{q^m}^{\times}}$, $0 \le i \le m-1$ such that
\[ (b^{q^k-1})^{\rho q^i} \delta^{\rho q^i}=\delta,  \]
for some $\rho \in \aut(\F_q)$.
\end{theorem}

As before, we obtain the following by restricting to linear automorphism group.

\begin{corollary}\label{cor:autGabtwis}
The linear automorphism group of $\mathcal{T}_{k,s,m}(\delta)$ is 
\[ \{ax \colon a \in \mathbb{F}_{q^k}\cap \mathbb{F}_{q^m}\}. \]
\end{corollary}

\section{Classification of $\mathbb{F}_{q^m}$-linear rank-metric codes of dimension $2$} \label{sec:3}

Sheekey in~\cite{sheekey2016new} established the first connection between rank-metric codes and linear sets and he pointed out a relation between $\F_{q^m}$-linear rank-metric codes and linear sets in the projective line $\PG(1,q^m)$.
By~\cite[Section 5]{sheekey2016new}, we have the following result.

\begin{theorem}\label{th:connJohn}
Let $\C=\langle f_1(x),f_2(x) \rangle_{\fqm}$ be a rank-metric code. Then $\C$ is an MRD code if and only if $U_{f_1,f_2}=\{ (f_1(x),f_2(x)) \colon x \in \fqm\}$ is a scattered $\fq$-linear subspace of dimension $m$ in $\fqm^2$ (or equivalently $L_{U_{f_1,f_2}}$ is scattered).
Moreover, two MRD codes $\C=\langle f_1(x),f_2(x) \rangle_{\fqm}$ and $\C'=\langle g_1(x),g_2(x) \rangle_{\fqm}$ are equivalent if and only if~$U_{f_1,f_2}$ and~$U_{g_1,g_2}$ are $\mathrm{\Gamma L}_2(q^m)$-equivalent. 
\end{theorem}

The previous result can be rephrased as follows in terms of linear equivalence of codes; see for instance~\cite{alfarano2022linear,randrianarisoa2020geometric}.

\begin{corollary}\label{cor:connJohn}
Let $\C=\langle f_1(x),f_2(x) \rangle_{\fqm}$ and $\C'=\langle g_1(x),g_2(x) \rangle_{\fqm}$ be MRD codes of dimension 2. We have that $\C$ and $\C'$ are linearly equivalent if and only if $U_{f_1,f_2}$ and $U_{g_1,g_2}$ are $\mathrm{G L}_2(q^m)$-equivalent. 
\end{corollary}

In~\cite{csajbok2018maximum}, Csajb\'ok and Zanella refined the classification of scattered $\fq$-linear subspaces of dimension $4$ in $\mathbb{F}_{q^4}^2$ provided by Bonoli and Polverino in~\cite{bonoli2005fq}. Using Theorem~\ref{th:connJohn}, the above results can be interpreted in terms of rank-metric code read as follows.

\begin{theorem}\label{th:q4allMRD}
Let $\C$ be an $\F_{q^4}$-linear MRD code in $\mathcal{L}_{4,q}[x]$ of dimension $2$. Then $\C$ is linearly equivalent to $\mathcal{T}_{2,1,4}(\delta)$ for some $\delta \in \mathbb{F}_{q^4}$ with $\mathrm{N}_{q^4/q}(\delta)\ne 1$. 
\end{theorem}
\begin{proof}
In~\cite[Theorem 3.4]{csajbok2018maximum} it has been proved that if $L_U$ is a scattered $\fq$-linear set of rank $4$ in $\mathrm{PG}(1,q^4)$, then it is $\mathrm{P\Gamma L}_2(q^4)$ to $L_{U'}$, where 
\[ U'=\{ (x,x^q+\delta x^{q^3}) \colon x \in \mathbb{F}_{q^4} \},\]
for some $\delta \in \F_{q^4}$ with $\mathrm{N}_{q^4/q}(\delta)\ne 1$. 
Since by~\cite[Section 4]{csajbok2018classes} $\fq$-linear sets of rank~$4$ are simple in $\mathrm{PG}(1,q^4)$ and also using~\cite[Corollary 4.4]{csajbok2018maximum}, we find that $U$ and $U'$ are $\mathrm{G L}_2(q^4)$-equivalent. Therefore, applying Corollary~\ref{cor:connJohn} we find that $\C$ is equivalent to the code $\langle x, x^q+\delta x^{q^3} \rangle_{\fqm}$.
\end{proof}

As mentioned in~\cite[Introduction]{bartoli2021conjecture}, scattered subspaces of dimension $m$ in $\mathbb{F}_{2^m}^2$ corresponds to translation hyperovals, which have been classified by Payne in~\cite{payne1971complete}.
As a consequence we obtain the following result in terms of MRD codes.

\begin{theorem}\label{th:MRDq=2}
Let $\C$ be an $\F_{2^m}$-linear MRD code in $\mathcal{L}_{m,2}[x]$ of either dimension $2$ or codimension $2$. Then $\C$ is linearly equivalent to  $\mathcal{G}_{2,s,m}$.
\end{theorem}
\begin{proof}
Let us first assume that $\C$ has dimension two.
Let $U$ be a subspace of $\mathbb{F}_{2^m}^2$ associated with $\C$, as in Theorem~\ref{th:connJohn}. Up to $\mathrm{GL}_2(2^m)$-equivalence, we may assume that 
\[ U=\{(x,f(x)) \colon x \in \F_{2^m}\},\]
for some $f(x)=\sum_{i=1}^{n-1} a_i x^{2^i}\in \mathcal{L}_{m,2}[x]$. Consider the graph of $f$ in $\mathrm{AG}(2,2^m)$
\[ \mathcal{G}(f)=\{(x , f(x)) \colon x \in \F_{2^m}\}, \]
which can be identified in the projective plane with $\{\langle (x , f(x),1)\rangle_{\F_{2^m}} \colon x \in \F_{2^m}\}$,
and the set of the determined directions by $f$
\[ \mathcal{D}(f)=\{\langle(x , f(x) , 0)\rangle_{\F_{2^m}} \colon x \in \F_{2^m}\} \subset \ell_{\infty}, \]
then the pointset 
\[ \mathcal{S}=\mathcal{G}(f) \cup (\ell_{\infty} \setminus \mathcal{D}(f)) \]
is a translation hyperoval of $\mathrm{PG}(2,2^m)$.
It is clear that $\mathcal{S}$ has size $2^m+2$. The line at infinity meet $\mathcal{S}$ in exactly two points since $\mathcal{D}(f)$ is contained in $\ell_{\infty}$ and has size~$2^m-1$ and hence its complement in $\ell_{\infty}$ is given by two points.
We can divide the affine lines into two classes: 
\begin{itemize}
\item the lines meeting $\ell_{\infty}\setminus \mathcal{D}(f)$ are $2$-secant to $\mathcal{S}$;
\item the lines meeting $\mathcal{D}(f)$ are either tangent or $2$-secant to $\mathcal{S}$.
\end{itemize}
Indeed, any line $\ell$ intersects the line at infinity in a point, namely $\langle (1,\alpha,0) \rangle_{\F_{2^m}}$ or~$\langle (0,1,0) \rangle_{\F_{2^m}}$.
Suppose that $\ell\cap \ell_{\infty} \notin \mathcal{D}(f)$, then $|\ell \cap \mathcal{G}(f)|=1$, since $|\ell \cap \mathcal{G}(f)|\leq1$ for any line $\ell$ such that $\ell\cap \ell_{\infty} \notin \mathcal{D}(f)$ and $\mathcal{G}(f)$ has size $2^m$. Therefore, $\ell \cap \mathcal{S}$ has size two.
Suppose that $\ell\cap \ell_{\infty}=\langle (1,\alpha,0) \rangle_{\F_{2^m}} \notin \mathcal{D}(f)$, then either $|\ell \cap \mathcal{G}(f)|=0$ or $|\ell \cap \mathcal{G}(f)|\geq 2$.
By contradiction, assume that there exist $u,v,z \in \F_{2^m}$ such that they are pairwise distinct and $\langle(u , f(u),1)\rangle_{\F_{2^m}}, \langle(v , f(v),1)\rangle_{\F_{2^m}} ,\langle(z , f(z),1)\rangle_{\F_{2^m}}  \in \ell$. 
Then $(u-v , f(u-v),0),(z-v,f(z-v),0) \in  \langle (1,\alpha,0) \rangle_{\F_{2^m}} \cap \mathcal{D}(f)$ and since $D(f)$ is scattered then $u-v$ and $z-v$ must be $\F_{2}$-proportional and hence $u=v$, a contradiction.
Therefore, in this case we have $|\ell \cap \mathcal{S}|\in \{0,2\}$.
By Theorem~\ref{th:classhyper}, $f(x)=a_j x^{2^j}$, for some positive integer $j\in \{1,\ldots,m-1\}$ such that $\gcd(j,m)=1$.
Therefore, we get $U=\{(x,a_j x^{2^j}) \colon x \in \F_{2^m}\}$ and $\C$ is equivalent to $\mathcal{G}_{2,j,m}$ (again by Theorem~\ref{th:connJohn}).
When the dimension of $\C$ is $m-2$, the result follows by a duality argument.
\end{proof}

\section{Counting the number of $\F_{q^m}$-linear MRD codes of dimension $2$} \label{sec:4}

In this section, we apply the results derived in the previous sections
to determine the number of $\F_{q^m}$-linear MRD codes in $\F_{q^m}^m$ of dimension $2$ in the cases $q=2$ and any~$m$, and
 $q>2$ and $m=4$.

We use the orbit-stabilizer theorem to count the number of certain families of $\fqm$-linear MRD codes. In detail, if we determine the number of the equivalence classes of the codes (say $\mathcal{C}_1,\ldots,\mathcal{C}_t$ is a set of representatives), then the number of codes we are looking for is
\begin{equation}\label{eq:orbstab} \sum_{i=1}^t \frac{|\GL|}{|\mathrm{Autlin}(\mathcal{C}_i)|}. \end{equation}

As a consequence, we derive some density results for these MRD codes. We define the density function of $\F_{q^m}$-linear rank-metric codes with parameters $[n,k,d]_{q^m/q}$ as the ratio of the number of codes with these parameters and the number of $\F_{q^m}$-linear subspaces of $\F_{q^m}^n$ having the same dimension, that is
$$\delta^\rk_d( m,n, k; q):=\frac{|\{\mC \le \mathbb{F}_{q^m}^n \mid \dim(\mC)=k, \, \drk(\mC) \ge d\}|}{\qmqbin{n}{k}}.$$
It is currently an open problem in coding theory to determine the exact proportion of MRD codes~\cite{antrobus2019maximal,byrne2020partition,gruica2022common,gruica2023rank,gruica2024densities,neri2018genericity}.

\subsection{Case $q=2$}

We start by observing that two generalized Gabidulin codes $\mathcal{G}_{k,s,m}$ and $\mathcal{G}_{k,t,m}$ are equivalent if and only if $s\equiv \pm t \pmod{m}$, by~\cite{lunardon2018generalized} and this is still true when considering the linear equivalence. Therefore, we have the following.

\begin{proposition}\label{prop:numberGab}
The number of linearly inequivalent generalized Gabidulin codes of a fixed dimension in $\mathcal{L}_{m,q}[x]$ is $\frac{\varphi(m)}2$.
\end{proposition}

Therefore, combining~\eqref{eq:orbstab}, Proposition~\ref{prop:numberGab}, Theorem~\ref{th:MRDq=2}, and Corollary~\ref{cor:autGab}, we obtain the number of $\mathbb{F}_{2^m}$-linear MRD codes of dimension two in $\mathcal{L}_{m,2}[x]$.

\begin{corollary}
The number of $\mathbb{F}_{2^m}$-linear MRD codes of dimension two in $\mathcal{L}_{m,2}$ is 
\[ \frac{\varphi(m)}2 \cdot \frac{|\mathrm{GL}_m(2)| }{(2^m-1)}. \]
\end{corollary}

The following result provides the density of MRD codes in $\F_{2^m}^m$ of dimension $2$ and its asymptotic.

\begin{corollary}
We have    
\[\delta^\rk_{m-1}( m,m, 2; 2)=\frac{2^{m^2} \cdot \varphi(m) \cdot \prod_{j=1}^m \left( 1-\frac{1}{2^j}\right)(2^{2m}-1)}{(2^{m^2}-1)(2^{m^2-m}-1)},\]
and
\[\delta^\rk_{m-1}( m,m, 2; 2)\in \mathcal{O}\left(m2^{-m^2+3m}\right) \textup{ as } m\longrightarrow+\infty.\]
\end{corollary}

\subsection{Case $m=4$ and $q>2$}

In the remainder of this section, we assume $q>2$. We derive bounds on the number of the equivalence classes of $\mathbb{F}_{q^4}$-linear MRD codes of dimension two in $\mathcal{L}_{4,q}[x]$, which arise from~\cite{csajbok2018maximum}.

\begin{proposition}\label{prop:numbMRDq4}
The number of linearly inequivalent $\mathbb{F}_{q^4}$-linear MRD codes of dimension two in $\mathcal{L}_{4,q}[x]$ is
$q(q-1)/2$.
\end{proposition}
\begin{proof}
Let start by noting that the number of linearly inequivalent $\mathbb{F}_{q^4}$-linear MRD codes of dimension two in $\mathcal{L}_{4,q}[x]$ is exactly the number of $\mathrm{PGL}_2(q^4)$-inequivalent scattered $\fq$-linear sets of rank $4$ in $\PG(1,q^4)$. This is indeed true because of Corollary~\ref{cor:connJohn} and as the scattered $\fq$-linear sets of rank-$4$ in $\PG(1,q^4)$ have $\mathrm{GL}$-class 1. 
Finally,~\cite[Theorem 4.5]{csajbok2018maximum} implies that there are exactly $q(q-1)/2$ linearly inequivalent $\mathbb{F}_{q^4}$-linear MRD codes of dimension two in $\mathcal{L}_{4,q}[x]$.
\end{proof}

We can now give the number of $\mathbb{F}_{q^4}$-linear MRD codes of dimension two in $\mathcal{L}_{4,q}[x]$.

\begin{corollary}
The number $M(q)$ of $\mathbb{F}_{q^4}$-linear MRD codes of dimension 2 in $\mathcal{L}_{4,q}[x]$ is 
\[ \frac{1}2 q^7 (q^3-1)(q^2-1)(q-1)(q^3-q^2-q-1). \]
\end{corollary}
\begin{proof}
By Corollary~\ref{cor:autGabtwis}, the linear automorphism group $\mathcal{T}_{2,1,4}(\delta)$ with $\delta \ne 0$
\[ \{ ax \colon a \in \mathbb{F}_{q^2}^* \} \]
and hence has size $q^2-1$.
If $\delta=0$, Corollary~\ref{cor:autGab} implies that the linear automorphism group is
\[ \{ ax \colon a \in \mathbb{F}_{q^4}^* \} \]
and hence it has size $q^4-1$.
By Proposition~\ref{prop:numbMRDq4}, the number of inequivalent $\mathbb{F}_{q^4}$-linear MRD codes of dimension 2 in $\mathcal{L}_{4,q}[x]$ is $q(q-1)/2$, of which one is the Gabidulin code and the remaining are twisted Gabidulin codes which are not equivalent to a Gabidulin code.
Therefore, we have by \eqref{eq:orbstab}
\[M(q) = \frac{|\mathrm{GL}_4(q)|}{q^4-1}+ \left(\frac{q(q-1)}2-1\right) \frac{|\mathrm{GL}_4(q)|}{q^2-1}, \]
where the first element of the sum corresponds to the orbit of $\mathcal{G}_{2,1,4}$.
The statement follows by straightforward computations.
\end{proof}

\begin{remark}\label{rem:w2}
    Applying this result to~\cite[Proposition~53]{cotardo2023rank}, we showed that $w_2(2,4,4;q)$ is a polynomial in $q$. A closed formula for $w_j(2,4,4;q)$, $j\in\{0,\ldots,4\}$, is given in~\cite[Theorem~59]{cotardo2023rank}.
\end{remark}

In the following result, we determine the density of 2-dimensional MRD codes in~$\F_{q^4}^4$ and its asymptotic behavior. At the time of writing this paper, this is one of the very few cases where the exact density can be computed.

\begin{corollary}
We have    
\[\delta^\rk_3( 4,4, 2; q)=\frac{1}2 \frac{q^7 (q^3-1)(q^2-1)(q-1)(q^3-q^2-q-1)(q^8-1)(q^4-1)}{(q^{16}-1)(q^{12}-1)}.\]
In particular,
\[\lim_{q\rightarrow \infty} \delta^\rk_3( 4,4, 2; q)=\frac{1}2.\]
\end{corollary}

The previous result shows that
when randomly selecting a 2-dimensional code in $\F_{q^4}^4$, one has approximately 50\% chance of picking an MRD code (for sufficiently large~$q$). 

\section{One-weight codes in $\F_{q^m}^{mk}$} \label{sec:5}

In this section we determine the number of one-weight codes within $\F_{q^m}^{mk}$. We first establish a connection between non-square rank-metric codes and multivariate linearized polynomials. We later use such a relation to compute the desired quantity.

\subsection{Non-square rank-metric codes}\label{sec:nonsquare}

As proved in~\cite{bartoli2022exceptional,polverino2023divisible}, any $\F_{q^m}$-linear rank-metric code in $\F_{q^m}^{m\ell}$ can be described as an $\F_{q^m}$-linear subspace of multivariate linearized polynomials. We let $\mathcal{L}_{m,q}[x_1,\ldots,x_{\ell}]$  (or $\mathcal{L}_{m,q}[\underline{x}]$) denote the $\F_{q^m}$-linear vector space of linearized polynomials over $\F_{q^m}$ in the indeterminates $\underline{x}=(x_1,\ldots,x_{\ell})$, that is the $\F_{q^m}$-span in $\F_{q^m}[\underline{x}]$ of
\[ \{ x_i^{q^j} \colon i \in \{1,\ldots,\ell\},j \in \mathbb{N}_0 \}. \]

Moreover, we let ${\mathcal{L}}_{m,q}[x_1,\ldots,x_{\ell}]$  (or ${\mathcal{L}}_{m,q}[\underline{x}]$)  be the $\F_{q^m}$-linear vector space of linearized polynomials $\mathcal{L}_{m,q}[\underline{x}]$ modulo the relations $x_1^{q^m}-x_1,\ldots,x_{\ell}^{q^m}-x_{\ell}$. It is known that as an $\fq$-linear vector space, $\mathcal{L}_{m,q}[\underline{x}]$ is $\fq$-linear isometric to the space of $\fq$-linear maps from $\F_{q^m}^\ell$ to $\F_{q^m}$ and to $\F_q^{m \times m\ell}$. Therefore, we define the rank of a $q$-polynomial $f \in {\mathcal{L}}_{m,q}[\underline{x}]$ as $\mathrm{rk}(f)=\dim_{\fq}(\mathrm{Im}(f))$. In particular, $\F_{q^m}$-linear rank-metric codes in $\F_{q^m}^{m\ell}$ correspond to $\F_{q^m}$-linear subspaces of ${\mathcal{L}}_{m,q}[\underline{x}]$.

\begin{proposition}\label{prop:ev}\cite[Proposition 4.5]{polverino2023divisible}
Let $\mathcal{B}=(a_1,\ldots,a_{m\ell})$ be an $\F_q$-basis of $\F_{q^m}^\ell$ (seen as an $\fq$-linear vector space).
The map
\begin{equation*}
    \begin{array}{cccc}
        ev_{\mathcal{B}} \colon & {\mathcal{L}}_{m,q}[x_1,\ldots,x_{\ell}] & \rightarrow & \F_{q^m}^{m\ell} \\
        & f &\mapsto &(f(a_1),\ldots,f(a_{\ell m}))
    \end{array}
\end{equation*}
is an $\F_{q^m}$-linear isomorphism which preserves the rank, that is $$\mathrm{rk}(f)=\dim_{\fq}(\langle f(a_1),\ldots,f(a_{\ell m}) \rangle_{\fq})\quad \textup{ for any } \mathcal{L}_{m,q}[x_1,\ldots,x_{\ell}].$$ 
Moreover, if $\mathcal{C}=\langle f_1(\underline{x}),\ldots,f_k(\underline{x}) \rangle_{\F_{q^m}}$, then a generator matrix for $ev_{\mathcal{B}}(\C)$ is 
\[ G=\begin{pmatrix}
f_1(a_1) & f_1(a_2) & \ldots & f_1(a_{\ell m})\\
f_2(a_1) & f_2(a_2) & \ldots & f_2(a_{\ell m})\\
\vdots & & & \\
f_k(a_1) & f_k(a_2) & \ldots & f_k(a_{\ell m})
\end{pmatrix}.\]
\end{proposition}

Therefore, the map $ev_{\mathcal{B}}$ allows us to study rank-metric codes as $\F_{q^m}$-linear subspaces of ${\mathcal{L}}_{m,q}[\underline{x}]$. The notion of equivalence of rank-metric codes can be also read in the context of linearized polynomials. Indeed, the linear equivalence of rank-metric codes $\F_{q}^{m\times \ell m}$ is given by the natural action of the following group on $\F_{q}^{m\times \ell m}$
\[ \{ (A,B) \colon A \in \mathrm{GL}_m(q),B\in \mathrm{GL}_{\ell m}(q) \}. \]
In terms of linearized polynomials, the elements of $\mathrm{GL}_{\ell m}(q)$ can be represented as invertible linear functions in $\F_{q^m}^\ell$ in itself and hence as elements of $({\mathcal{L}}_{m,q} [\underline{x}])^\ell$. It follows that equivalent rank-metric codes can be defined via the natural action of the following group on ${\mathcal{L}}_{m,q}[\underline{x}]$ 
\[ \{ (g,f) \colon g \in {\mathcal{L}}_{m,q}[x], f=(f_1,\ldots,f_{\ell})\in({\mathcal{L}}_{m,q} [\underline{x}])^\ell \textup{ and } f,g \text{ invertible} \}. \]

Therefore, $\C_1, \C_2 \le \cL_{\ell m,q}[\underline{x}]$ are \textbf{equivalent} (or \textbf{linearly equivalent})
if there exist invertible \smash{$q$-polynomials} $f \in \cL_{m,q}[x]$, $g=(g_1,\ldots,g_{\ell})^{\ell} \in \cL_{m,q}[\underline{x}]$ and a field automorphism \smash{$\rho \in \mathrm{Aut}(\fq)$} such that
\[ \C_1=f \circ \C_2^\rho \circ g, \]
($\C_1=f \circ \C_2^\rho \circ g$, respectively).
The \textbf{automorphism group} of $\mC$ is
\[ \mathrm{Aut}(\C)=\{ (g,\rho,f) \in \mathrm{GL}_m(q)\times \mathrm{Aut}(\fq)\times \mathrm{GL}_{\ell m}(q) \colon g \circ \C^\rho \circ f  \}, \]
and the \textbf{linear automorphism group} of $\mC$ is
\[ \mathrm{Autlin}(\C)=\{ f \in \mathrm{GL}_{\ell m}(q) \colon \C \circ f=\C  \}. \]

\subsection{Counting one-weight codes}

The goal of this section is to establish the number of $[mk,k,m]_{\fqm}$-codes in $\F_{q^m}^{mk}$. We recall the following result.

\begin{proposition}(\cite[Proposition 3.6]{alfarano2022linear})
Let $k\geq 2$ and let $\C$ be a $[km,k]_{\fqm}$-code. We have that $\C$ is a one-weight code (with weight $m$) if and only if $\C$ is linearly equivalent to a code with generator matrix
\[ (I_k\mid\alpha I_k\mid\cdots \mid \alpha^{m-1} I_k), \]
for some $\alpha \in \F_{q^m}$ such that $\F_{q^m}=\fq(\alpha)$.
\end{proposition}

We start by determining the equivalence classes of one-weight codes.

\begin{proposition}\label{prop:oneclass}
We have only one equivalence class of $[mk,k,m]_{q^m/q}$-codes in $\F_{q^m}^{mk}$.
\end{proposition}
\begin{proof}
Let $\alpha,\beta \in \F_{q^m}$ such that $\F_{q^m}=\fq(\alpha)$ and consider the codes $\C$ and $\C'$ having as a generator matrix
\[ G=(I_k\mid\alpha I_k\mid\cdots \mid \alpha^{m-1} I_k)\,\,\,\text{and}\,\,\,G=(I_k\mid\beta I_k\mid\cdots \mid \beta^{m-1} I_k),  \]
respectively.
For any $i \in \{1,\ldots,m-1\}$, write
\[ \beta^i =a_{0,i}+a_{1,i} \alpha+\ldots+a_{m-1,i}\alpha^{m-1},\]
with $a_0,\ldots,a_{m-1} \in \fq$.
Consider the block-matrix
\[
A=\begin{pmatrix}
a_{0,0} I_k & a_{0,1} I_k & \cdots & a_{0,m-1} I_k \\
a_{1,0} I_k & a_{1,1} I_k & \cdots & a_{1,m-1} I_k \\
\vdots &  &  & \\
a_{m-1,0} I_k & a_{m-1,1} I_k & \cdots & a_{m-1,m-1} I_k \\
\end{pmatrix} \in \mathbb{F}_{q}^{mk\times mk},
\]
and note that its determinant is equal to the determinant of the matrix $(a_{ij}) \in \mathbb{F}_q^{m\times m}$ (since it is the Kronecker product of this matrix with the identity matrix), which is invertible since it is the matrix change of basis between the ordered bases~$(1,\alpha,\ldots,\alpha^{m-1})$ and $(1,\beta,\ldots,\beta^{m-1})$.
Therefore, $A \in \mathrm{GL}_{km}(q)$.
Also, $GA$ is equal to
\begin{multline*}
    ( (a_{0,0}+a_{1,0}\alpha+\ldots+a_{m-1,0}\alpha^{m-1} )I_k\mid (a_{0,1}+a_{1,1}\alpha+\ldots+a_{m-1,1}\alpha^{m-1} )I_k \mid \cdots \\\mid (a_{0,m-1}+a_{1,m-1}\alpha+\ldots+a_{m-1,m-1}\alpha^{m-1} )I_k )=G',
\end{multline*}

that is $\C$ and $\C'$ are linearly equivalent.
\end{proof}

\begin{remark}
The above result can also be proved by using the geometric interpretation of $\F_{q^m}$-linear rank-metric codes established in~\cite{randrianarisoa2020geometric}, see also~\cite{alfarano2022linear}.
\end{remark}

In order to determine the number of $[mk,k,m]_{q^m/q}$-codes in $\F_{q^m}^{mk}$, we first determined the related linear automorphism group.
Based on Section~\ref{sec:nonsquare} and on Proposition~\ref{prop:oneclass}, we may only consider the code 
\[\C=\langle x_1,\ldots,x_k \rangle_{\F_{q^m}} \subseteq \mathcal{L}_{m,q}[x_1,\ldots,x_k],\]
since it is an $[mk,k,m]_{\fqm}$-code.

\begin{theorem}\label{thm:autgrouponeweight}
Let $k\geq 2$.
The linear automorphism group of  $\C=\langle x_1,\ldots,x_k \rangle_{\F_{q^m}}$ is
\[\mathrm{Autlin}(\C)=\{ (a_{1,1}x_1+\ldots+a_{1,k}x_k,\ldots,a_{k,1}x_1+\ldots+a_{k,k}x_k) \colon \mathrm{rk}((a_{i,j}))=k \}.\]
In particular, we get
\[ |\mathrm{Autlin}(\C)|=|\mathrm{GL}_{q^m}(k)|= (q^{mk}-1)(q^{mk}-q^m)\cdots (q^{mk}-q^{m(k-1)}). \]
\end{theorem}
\begin{proof}
Recall that
\[ \mathrm{Autlin}(\C)=\{ f \in \mathrm{GL}_{k m}(q) \colon \C \circ f=\C  \}=\{ (f_1,\ldots,f_k)\in \mathrm{GL}_{k m}(q) \colon x_i \circ f_i \in \C,\,\,\forall i\in[k] \}. \]
Therefore, we have that there exist $a_{i,j} \in \mathbb{F}_{q^m}$ such that 
\[f_i(\underline{x})=a_{i,1}x_1+\ldots+a_{i,k}x_k\]
for any $i \in [k]$. Since the map defined by $(f_1,\ldots,f_k)$ is invertible if and only if the matrix~$\smash{(a_{i,j})}$ is invertible, the result follows.
\end{proof}

The following result follows by combining the orbit-stabilizer theorem, Proposition~\ref{prop:oneclass} and Theorem~\ref{thm:autgrouponeweight}.

\begin{corollary}\label{cor:oneweightcode}
The number of  $[mk,k,m]_{\fqm}$-codes is
\[ \frac{|\mathrm{GL}_{mk}(q)|}{(q^{mk}-1)(q^{mk}-q^m)\cdots (q^{mk}-q^{m(k-1)})}.\]
\end{corollary}

As a consequence, we can determine the density of MRD codes in $\F_{q^m}^{mk}$ of dimension~$k$, as well as their asymptotic behaviour.

\begin{corollary}
We have    
\[\delta^\rk_m( m,mk, k; q)= \frac{|\mathrm{GL}_{mk}(q)|}{(q^{mk}-1)(q^{mk}-q^m)\cdots (q^{mk}-q^{m(k-1)})\qmqbin{mk}{k}},\]
and
\[\lim_{q\rightarrow \infty} \delta^\rk_m( m,mk, k; q)=1.\]
\end{corollary}

\section{Explicit results on Whitney numbers} \label{sec:6}

In this section, we establish some numerical results on the Whitney numbers of rank-metric lattices. In some cases we are able to provide closed formulas for them, which is in general a very hard task. We obtain these results as corollaries of the formulas established throughout the paper.

The following result is a consequence of~\cite[Corollary~3.1]{ravagnani2022whitney} and Corollary~\ref{cor:oneweightcode}.

\begin{corollary}\label{prop:i2m3}
    Let $i=2$ and $m=3$. For $n\in\{4,5\}$ and $j\in\{1,\ldots,n\}$ we have
    \begin{align*}
        w_j(2,n,3;q)=\qqbin{n}{3}\qmqbin{n-1}{j-1}(q^3-q)(q^3-q^2)(-1)^{j-1}q^{3\binom{j-1}{2}}.
    \end{align*}
    Moreover, for $n=6$ and $j\in\{1,\ldots,6\}$ we get
    \begin{align*}
        w_j(2,6,3;q)=&\qqbin{6}{3}\qmqbin{5}{j-1}(q^3-q)(q^3-q^2)(-1)^{j-1}q^{\binom{j-1}{2}}\\+&\qmqbin{4}{j-2}(q^6-q)(q^6-q^2)(q^6-q^4)(q^6-q^5)(-1)^{j-2}q^{3\binom{j-2}{2}}.
    \end{align*}
\end{corollary}
\begin{proof}
    For $n\in\{4,5\}$, we have
    \begin{align*}
        w_j(2,n,3;q)=|\{\left<x\right>\in\fqm \colon \rk(x)=3\}|\qmqbin{n-1}{j-1}(-1)^{j-1}q^{3\binom{j-1}{2}}
    \end{align*}
    which implies the first part of the statement. For $n=6$,~\cite[Corollary~3.1]{ravagnani2022whitney}  implies
    \begin{multline*}
        w_j(2,6,3;q)=|\{\left<x\right>\in\fqm\colon \rk(x)=3\}|\qmqbin{5}{j-1}(-1)^{j-1}q^{3\binom{j-1}{2}}\\
        +|\{C\in\fqm\colon C \textup{ is } [6,2,3]_{\mathbb{F}_{q^3}}\textup{-code}\}|\qmqbin{4}{j-2}(-1)^{j-2}q^{3\binom{j-2}{2}}.
    \end{multline*}
    The statement now follows from Corollary~\ref{cor:oneweightcode}.
\end{proof}

 The next result is the analogue of~\cite[Theorem~5.1]{ravagnani2022whitney} and provides a recursive formula for computing the Whitney numbers of the first kind.

\begin{theorem}\label{prop:recursion}
    If $1\leq ij\leq\min\{n,m\}$ then 
    \begin{equation*}
        w_j(i,n,m;q)=\sum_{s=1}^{ij}\qqbin{n}{s}\sum_{t=1}^sw_j(i,t,m;q)\qqbin{s}{t}q^{\binom{s-t}{2}}(-1)^{s-t}.
    \end{equation*}
\end{theorem}
\begin{proof}
    Fix $j\in\{1,\ldots,n\}$ and notice that, for all $X\in\mL_i(n,m;q)$ with $\dim(X)=j$, we have $\dim(\supp(X))\leq ij$. We get
    \begin{align*}
        w_j(i,n,m;q)&=\sum_{\substack{X\in\mL_i(n,m;q)\\\dim(X)=j}}\mu_{n,m,q}^{(i)}(X)\\&=\sum_{s=0}^{ij}\sum_{\substack{S\leq\fq^n\\\dim(S)=s}}\sum_{\substack{X\in\mL_i(n,m;q)\\\dim(X)=j\\\supp(X)=S}}\mu_{n,m,q}^{(i)}(X)\\&=\sum_{s=0}^{ij}\sum_{\substack{S\leq\fq^n\\\dim(S)=s}}f(S)
    \end{align*}
    where, for all $S\leq\fq^n$, we define
    \begin{equation*}
        f(S)=\sum_{\substack{X\in\mL_i(n,m;q)\\\dim(X)=j\\\supp(X)=S}}\mu_{n,m,q}^{(i)}(X).
    \end{equation*}
    For a subspace $S\leq\fq^n$ and for all $T\leq S$, define
    \begin{equation*}
        g(T)=\sum_{T'\leq T}f(T').
    \end{equation*}
    
    Since $ij\geq 1$ we have $j\geq 1$ and $g(\left<0\right>)=0$ (sum over an empty index set). Moreover, observe that, for any $X\in\mL_i(n,m;q)$ with $\supp(X)= S$ there exists a lattice isomorphism that maps the interval $[0,X]$ in $\mL_i(n,m;q)$ to $\mL_i(\dim(S),m;q)$. Therefore, for any $t$-dimensional $T\leq S$, we get
    \begin{equation*}
        g(T)=\sum_{\substack{X\in\mL_i(n,m;q)\\\dim(X)=j\\\supp(X)\leq T}}\mu_{n,m,q}^{(i)}(X)=w_j(i,t,m;q).
    \end{equation*}
    If $s=\dim(S)$, using M\"obius inversion, then we get
    \begin{equation*}
        f(S)=\sum_{T\leq S}g(T)(-1)^{s-\dim(T)}=\sum_{t=1}^sw_j(i,t,m;q)\qqbin{s}{t}q^{\binom{s-t}{2}}(-1)^{s-t}.
    \end{equation*}
    Therefore, we obtain
    \begin{align*}
        w_j(i,n,m;q)&=\sum_{s=1}^{ij}\sum_{\substack{S\leq\fq^n\\\dim(S)=s}}f(S)\\
        &=\sum_{s=1}^{ij}\sum_{\substack{S\leq\fq^n\\\dim(S)=s}}\sum_{t=1}^sw_j(i,t,m;q)\qqbin{s}{t}q^{\binom{s-t}{2}}(-1)^{s-t}\\
        &=\sum_{s=1}^{ij}\qqbin{n}{s}\sum_{t=1}^sw_j(i,t,m;q)\qqbin{s}{t}q^{\binom{s-t}{2}}(-1)^{s-t},        
    \end{align*}
which concludes the proof.   
\end{proof}

One one the main difference with the Hamming-metric is the computation of the quantity $f(\fqn)$, which could be explicitly determined in~\cite[Theorem~5.1]{ravagnani2022whitney}.

As an immediate consequence of Theorem~\ref{prop:recursion},  we can compute the Whitney numbers of the first kind of the lattices $\mL_2(n,m;q)$ for $m\in\{3,4\}$.

\begin{corollary}
    For any $n\geq 2$ and  $j\in\{1,2,3\}$, we have
    \begin{equation*}
        w_j(2,n,3;q)=\sum_{s=1}^{2j}\qqbin{n}{s}\sum_{t=1}^sw_j(2,t,3;q)\qqbin{s}{t}q^{\binom{s-t}{2}}(-1)^{s-t},
    \end{equation*}
    where the values of the $w_j(2,t,3;q)$'s is given in Proposition~\ref{prop:i2m3}.
\end{corollary}

\begin{corollary}
   For any $n\geq 2$ and  $j\in\{1,2,3,4\}$, we have
    \begin{equation*}
w_j(2,n,4;q)=\sum_{s=1}^{2j}\qqbin{n}{s}\sum_{t=1}^sw_j(2,t,4;q)\qqbin{s}{t}q^{\binom{s-t}{2}}(-1)^{s-t},
    \end{equation*}
    where the values of $w_j(2,4,4;q)$'s are given in~\cite[Theorem~59]{cotardo2023rank}.
\end{corollary}

\bigskip
\bibliographystyle{abbrv}
\bibliography{ourbib}

\begin{thebibliography}{10}

\bibitem{alfarano2022linear}
G.~N. Alfarano, M.~Borello, A.~Neri, and A.~Ravagnani.
\newblock Linear cutting blocking sets and minimal codes in the rank metric.
\newblock {\em Journal of Combinatorial Theory, Series A}, 192:105658, 2022.

\bibitem{antrobus2019maximal}
J.~Antrobus and H.~Gluesing-Luerssen.
\newblock Maximal {F}errers diagram codes: {C}onstructions and genericity
  considerations.
\newblock {\em IEEE Transactions on Information Theory}, 65(10):6204--6223,
  2019.

\bibitem{bartoli2021conjecture}
D.~Bartoli, B.~Csajb{\'o}k, and M.~Montanucci.
\newblock On a conjecture about maximum scattered subspaces of
  $\mathbb{F}_{q^6}\times \mathbb{F}_{q^6}$.
\newblock {\em Linear Algebra and its Applications}, 631:111--135, 2021.

\bibitem{bartoli2022exceptional}
D.~Bartoli, G.~Marino, A.~Neri, and L.~Vicino.
\newblock Exceptional scattered sequences.
\newblock {\em arXiv preprint: 2211.11477}, 2022.

\bibitem{blokhuis2000scattered}
A.~Blokhuis and M.~Lavrauw.
\newblock Scattered spaces with respect to a spread in $\textnormal{PG}(n,q)$.
\newblock {\em Geometriae Dedicata}, 81(1):231--243, 2000.

\bibitem{bonoli2005fq}
G.~Bonoli and O.~Polverino.
\newblock $\mathbb{F}_q$-linear blocking sets in $\mathrm{PG}(2,q^4)$.
\newblock {\em Innovations in Incidence Geometry: Algebraic, Topological and
  Combinatorial}, 2(1):35--56, 2005.

\bibitem{byrne2020partition}
E.~Byrne and A.~Ravagnani.
\newblock Partition-balanced families of codes and asymptotic enumeration in
  coding theory.
\newblock {\em Journal of Combinatorial Theory, Series A}, 171:105169, 2020.

\bibitem{caullery2015classification}
F.~Caullery and K.-U. Schmidt.
\newblock On the classification of hyperovals.
\newblock {\em Advances in Mathematics}, 283:195--203, 2015.

\bibitem{cherowitzo1996hyperovals}
W.~Cherowitzo.
\newblock Hyperovals in desarguesian planes: an update.
\newblock {\em Discrete Mathematics}, 155(1-3):31--38, 1996.

\bibitem{cooperstein1998external}
B.~N. Cooperstein.
\newblock External flats to varieties in pg$(\lambda^2(v))$ over finite fields.
\newblock {\em Geometriae Dedicata}, 69:223--235, 1998.

\bibitem{cotardo2023rank}
G.~Cotardo and A.~Ravagnani.
\newblock Rank-metric lattices.
\newblock {\em The Electronic Journal of Combinatorics}, pages P1--4, 2023.

\bibitem{csajbok2018classes}
B.~Csajb{\'o}k, G.~Marino, and O.~Polverino.
\newblock Classes and equivalence of linear sets in {PG}$(1, q^n)$.
\newblock {\em Journal of Combinatorial Theory, Series A}, 157:402--426, 2018.

\bibitem{csajbok2018maximum}
B.~Csajb{\'o}k and C.~Zanella.
\newblock Maximum scattered $\mathbb{F}_q$-linear sets of $\mathrm{PG}(1,q^4)$.
\newblock {\em Discrete Mathematics}, 341(1):74--80, 2018.

\bibitem{de2002arcs}
M.~J. De~Resmini, D.~Ghinelli, and D.~Jungnickel.
\newblock Arcs and ovals from abelian groups.
\newblock {\em Designs, Codes and Cryptography}, 26(1):213--228, 2002.

\bibitem{delsarte1978bilinear}
P.~Delsarte.
\newblock Bilinear forms over a finite field, with applications to coding
  theory.
\newblock {\em Journal of Combinatorial Theory, Series A}, 25(3):226--241,
  1978.

\bibitem{dowling1971codes}
T.~A. Dowling.
\newblock Codes, packings and the critical problem.
\newblock In A.~Barlotti, editor, {\em Atti del Convegno di Geometria
  Combinatoria e sue Applicazioni}, pages 209--224. 1971.

\bibitem{dowling1973q}
T.~A. Dowling.
\newblock A {$q$}-analog of the partition lattice.
\newblock In {\em A Survey of Combinatorial Theory}, pages 101--115. 1973.

\bibitem{gabidulin1985theory}
E.~M. Gabidulin.
\newblock Theory of codes with maximum rank distance.
\newblock {\em Problemy peredachi informatsii}, 21(1):3--16, 1985.

\bibitem{gruica2024densities}
A.~Gruica, A.-L. Horlemann, A.~Ravagnani, and N.~Willenborg.
\newblock Densities of codes of various linearity degrees in
  translation-invariant metric spaces.
\newblock {\em Designs, Codes and Cryptography}, 92(3):609--637, 2024.

\bibitem{gruica2022common}
A.~Gruica and A.~Ravagnani.
\newblock Common complements of linear subspaces and the sparseness of mrd
  codes.
\newblock {\em SIAM Journal on Applied Algebra and Geometry}, 6(2):79--110,
  2022.

\bibitem{gruica2023rank}
A.~Gruica, A.~Ravagnani, J.~Sheekey, and F.~Zullo.
\newblock Rank-metric codes, semifields, and the average critical problem.
\newblock {\em SIAM Journal on Discrete Mathematics}, 37(2):1079--1117, 2023.

\bibitem{hirschfeld1975ovals}
J.~Hirschfeld.
\newblock Ovals in desarguesian planes of even order.
\newblock {\em Annali di Matematica Pura ed Applicata}, 102(1):79--89, 1975.

\bibitem{hirschfeld1979projective}
J.~Hirschfeld.
\newblock Projective geometry over finite fields.
\newblock {\em Oxford math. Monographs}, 1979.

\bibitem{kshevetskiy2005new}
A.~Kshevetskiy and E.~Gabidulin.
\newblock The new construction of rank codes.
\newblock In {\em Proceedings. International Symposium on Information Theory,
  2005. ISIT 2005.}, pages 2105--2108. IEEE, 2005.

\bibitem{lavrauw2015field}
M.~Lavrauw and G.~Van~de Voorde.
\newblock Field reduction and linear sets in finite geometry.
\newblock {\em Topics in finite fields}, 632:271--293, 2015.

\bibitem{liebhold2016automorphism}
D.~Liebhold and G.~Nebe.
\newblock Automorphism groups of {G}abidulin-like codes.
\newblock {\em Archiv der Mathematik}, 107(4):355--366, 2016.

\bibitem{lunardon2017kernels}
G.~Lunardon, R.~Trombetti, and Y.~Zhou.
\newblock On kernels and nuclei of rank metric codes.
\newblock {\em Journal of Algebraic Combinatorics}, 46(2):313--340, 2017.

\bibitem{lunardon2018generalized}
G.~Lunardon, R.~Trombetti, and Y.~Zhou.
\newblock Generalized twisted gabidulin codes.
\newblock {\em Journal of Combinatorial Theory, Series A}, 159:79--106, 2018.

\bibitem{neri2018genericity}
A.~Neri, A.-L. Horlemann-Trautmann, T.~Randrianarisoa, and J.~Rosenthal.
\newblock On the genericity of maximum rank distance and {G}abidulin codes.
\newblock {\em Designs, Codes and Cryptography}, 86(2):341--363, 2018.

\bibitem{neri2020equivalence}
A.~Neri, S.~Puchinger, and A.-L. Horlemann-Trautmann.
\newblock Equivalence and characterizations of linear rank-metric codes based
  on invariants.
\newblock {\em Linear Algebra and its Applications}, 603:418--469, 2020.

\bibitem{payne1971complete}
S.~E. Payne.
\newblock A complete determination of translation ovoids in finite desarguian
  planes.
\newblock {\em Atti della Accademia Nazionale dei Lincei. Classe di Scienze
  Fisiche, Matematiche e Naturali. Rendiconti}, 51(5):328--331, 1971.

\bibitem{polverino2010linear}
O.~Polverino.
\newblock Linear sets in finite projective spaces.
\newblock {\em Discrete Mathematics}, 310(22):3096--3107, 2010.

\bibitem{polverino2023divisible}
O.~Polverino, P.~Santonastaso, J.~Sheekey, and F.~Zullo.
\newblock Divisible linear rank metric codes.
\newblock {\em IEEE Transactions on Information Theory}, 2023.

\bibitem{randrianarisoa2020geometric}
T.~H. Randrianarisoa.
\newblock A geometric approach to rank metric codes and a classification of
  constant weight codes.
\newblock {\em Designs, Codes and Cryptography}, 88(7):1331--1348, 2020.

\bibitem{ravagnani2022whitney}
A.~Ravagnani.
\newblock Whitney numbers of combinatorial geometries and higher-weight dowling
  lattices.
\newblock {\em SIAM Journal on Applied Algebra and Geometry}, 6(2):156--189,
  2022.

\bibitem{roth1991maximum}
R.~M. Roth.
\newblock Maximum-rank array codes and their application to crisscross error
  correction.
\newblock 37(2):328--336, 1991.

\bibitem{schmidt2018number}
K.-U. Schmidt and Y.~Zhou.
\newblock On the number of inequivalent gabidulin codes.
\newblock {\em Designs, Codes and Cryptography}, 86(9):1973--1982, 2018.

\bibitem{segre1955ovals}
B.~Segre.
\newblock Ovals in a finite projective plane.
\newblock {\em Canadian Journal of Mathematics}, 7:414--416, 1955.

\bibitem{sheekey2016new}
J.~Sheekey.
\newblock A new family of linear maximum rank distance codes.
\newblock {\em Advances in Mathematics of Communications}, 10(3):475--488,
  2016.

\bibitem{stanley2011enumerative}
R.~Stanley.
\newblock {\em Enumerative Combinatorics}, volume~1.
\newblock Cambridge University Press, 2nd edition, 2011.

\bibitem{wu2013linearized}
B.~Wu and Z.~Liu.
\newblock Linearized polynomials over finite fields revisited.
\newblock {\em Finite Fields and Their Applications}, 22:79--100, 2013.

\end{thebibliography}

\end{document}